\documentclass[leqno,12pt]{amsart}

\usepackage{amssymb,amsmath,multirow,graphics,multicol}

\theoremstyle{plain}
\newtheorem{theorem}{Theorem}[section]

\newtheorem{proposition}{Proposition}[section]
\newtheorem{corollary}{Corollary}[section]

\numberwithin{equation}{section}

\theoremstyle{remark}
\newtheorem{remark}{Remark}[section]
 \numberwithin{equation}{section}

\def\<{\left < }
\def\>{\right >}
\def\({\left ( }
\def\){\right )}
\def\e{\eqref}
\def\x{{\bf x}}

\def\x{\hbox{\bf x}}

\begin{document}
\title[Incompressible canonical vector field]{Euclidean submanifolds with incompressible canonical vector field}
\author[B.-Y. Chen]{Bang-Yen Chen} \address{2231 Tamarack Drive, Okemos, Michigan 48864-5929, U.S.A.} \email{chenb@msu.edu}

\begin{abstract} For a submanifold $M$ in a Euclidean space $\mathbb E^m$,
the tangential component $\hbox{\bf x}^T$ of the position vector field $\x$ of $M$ is the most natural vector field tangent to the Euclidean submanifold, called the \textit{canonical vector field} of $M$. 
In this article,  first we prove that the canonical vector field of every Euclidean submanifold is always conservative. Then we initiate the study of Euclidean submanifolds with incompressible canonical vector fields. In particular, we obtain the necessary and sufficient conditions for the canonical vector field of a Euclidean submanifold to be incompressible. Further, we provide examples of Euclidean submanifolds with incompressible canonical vector field. Moreover, we classify planar curves, surfaces of revolution and hypercylinders with incompressible canonical vector fields.

\end{abstract}

\maketitle

\section{Introduction}

For an $n$-dimensional submanifold $M^n$ in the Euclidean $m$-space $\mathbb{E}
^m$, the most elementary and natural geometric object is the position vector
field $\mathbf{x}$ of $M^n$. The position vector is a Euclidean vector $
\mathbf{x} =\overrightarrow{OP}$ that represents the position of a point $
P\in M^n$ in relation to an arbitrary reference origin $O\in \mathbb{E}^m$.

The position vector field plays important roles in physics, in particular in
mechanics. For instance, in any equation of motion, the position vector $\mathbf{x}(t)$ is usually the most sought-after quantity because the
position vector field defines the motion of a particle (i.e., a point mass):
its location relative to a given coordinate system at some time variable $t$. The first and the second derivatives of the position vector field with
respect to time $t$ give the velocity and acceleration of the particle.

For the Euclidean submanifold $M^n$, there exists a natural
decomposition of the position vector field $\mathbf{x}$ given by: 
\begin{align}
\mathbf{x}=\mathbf{x}^T+\mathbf{x}^N,
\end{align}
where $\mathbf{x}^T$ and $\mathbf{x}^N$ denote the tangential and the normal
components of $\mathbf{x}$, respectively. 

A vector field on a Riemannian manifold is called {\it conservative} if it is  the gradient of some function, known as a {\it scalar potential}. Conservative vector fields appear naturally in mechanics: They are vector fields representing forces of physical systems in which energy is conserved. 
Conservative vector fields have the property that the line integral is path independent, i.e., the choice of any path between two points does not change the value of the line integral   (cf. e.g., \cite{MT03}).   

A vector field on a Riemannian manifold  is called \textit{incompressible}  if it is a vector field with divergence zero at all points in the field. An important class of incompressible vector fields are magnetic fields. It is well-known that magnetic fields are widely used throughout modern technology, particularly in electrical engineering and electromechanics  (cf. e.g., \cite{A89}). 

In earlier articles, we have investigated Euclidean submanifolds whose canonical vector fields are concurrent \cite{C16,C17a}, concircular \cite {CW17},  torse-forming \cite{CV17}, or conformal \cite{CD17}. (See \cite{C17c,C17d} for recent surveys on several topics associated with position vector fields on Euclidean submanifolds.)

In this article,  first we prove that the canonical vector field of every Euclidean submanifold is always conservative. Then we initiate the study of Euclidean submanifolds with incompressible canonical vector fields. In particular, we obtain the necessary and sufficient conditions for the canonical vector field of a Euclidean submanifold to be incompressible. Further, we provide examples of Euclidean submanifolds with incompressible canonical vector field. Moreover, we classify planar curves, surfaces of revolution and hypercylinders with incompressible canonical vector fields.

\section{Preliminaries}
\label{section 2}

Let $x: M^n\to \mathbb{E}^m$ be an isometric immersion of a connected
Riemannian manifold $M^n$ into a Euclidean $m$-space $\mathbb{E}^m$. For each
point $p\in M^n$, we denote by $T_pM^n$ and $T^\perp_p M^n$ the tangent space and
the normal space of $M^n$ at $p$, respectively.
Let $\nabla$ and $\tilde\nabla$ denote the Levi--Civita connections of $M^n$
and ${\mathbb{E}}^{m}$, respectively. The formulas of Gauss and Weingarten
are given respectively by (cf. \cite{book73,book11,book17}) 
\begin{align}  \label{2.1}
&\tilde \nabla_XY=\nabla_X Y+h(X,Y), \\
& \tilde\nabla_X\xi =-A_\xi X+D_X\xi ,  \label{2.2}
\end{align}
for vector fields $X,\,Y$ tangent to $M$ and $\xi$ normal to $M^n$, where $h$
denotes the second fundamental form, $D$ the normal connection and $A$ the shape
operator of $M^n$.

For each normal vector $\xi$ at $p$, the shape operator $A_\xi$ is a
self-adjoint endomorphism of $T_pM^n$. The second fundamental form $h$ and the
shape operator $A$ are related by 
\begin{equation}  \label{2.3}
g(A_\xi X,Y)g=\tilde g(h(X,Y), \xi),
\end{equation}
where $g$ and $\tilde g$ denote the metric of $M$ and the metric of the
ambient Euclidean space, respectively.

The mean curvature vector $H$ of an $n$-dimensional submanifold $M^n$ is
defined by 
\begin{align}  \label{2.4}
H=\(\frac{1}{n}\) \mathrm{trace}\; h.
\end{align}

The Laplacian $\Delta $ of $M^n$ acting on smooth vector fields on a
Riemannian $n$-manifold $(M^n,g)$ is defined by
\begin{align}  \label{2.5}
\Delta X=-\sum\limits_{i=1}^{n}\left( \nabla _{e_{i}}\nabla _{e_{i}}X-\nabla
_{\nabla _{e_{i}}e_{i}}X\right),
\end{align}
where $\{e_1,\ldots,e_n\}$ is an orthonormal local frame of $M^n$.

\section{Canonical vector field of a Euclidean submanifold}

For the canonical vector field $\x^T$ of a Euclidean submanifold $M^n$ in the Euclidean $m$-space $\mathbb E^m$, we have the following general results.

\begin{theorem}\label{T:3.1} Let $M^n$ be a submanifold of the Euclidean $m$-space $\mathbb{E}^m$.
Then we have:
\begin{itemize}
\item[\textrm{(1)}] The canonical vector field  of $M^n$ is always conservative.

\item[\textrm{(2)}] The canonical vector field  of $M^n$ is incompressible if and only if $\<H,\x\>=-1$  holds identically on $M^n$.
\end{itemize}
\end{theorem}
\begin{proof} Assume that $M^n$ is an $n$-dimensional submanifold of $\mathbb{E}^{m}$. We put
\begin{align}\label{3.1} f=\frac{1}{2} \<\x,\x\>,\end{align}
where $\x$ denotes the positive vector field of $M^n$ in $\mathbb E^m$.
Let $\{e_1,\ldots,e_n\}$ be an orthonormal  local frame of $M^n$. Then it follows from \e{3.1} that the gradient of $f$ satisfies
\begin{equation}\begin{aligned}\label{3.2}
 \nabla f\,&=\frac{1}{2}\sum_{i=1}^n (e_i \<\x,\x\>)e_i =\sum_{i=1}^n \<\right.\hskip-.02in \tilde\nabla_{e_i}\x,\x\hskip-.02in \left.\> e_i
\\& = \sum_{i=1}^n \<e_i,\x\>e_i=\x^T,
\end{aligned}\end{equation}
which implies that the canonical vector field is conservative with scalar potential $f$. This proves statement (1).

To prove statement (2), we need to compute the divergence ${\rm div} (\x^T)$ of the canonical vector field $\x^T$ of $M^n$. From \e{2.1}, \e{2.4}, \e{3.2} and the definition of divergence (cf. e.g. \cite{book11,book17}), we find
\begin{equation}\begin{aligned}\label{3.3}& \hskip-.1in 
{\rm div} (\x^T)=\sum_{i=1}^n \< \nabla_{e_i}\x^T,e_i \>=\sum_{i=1}^n \< \nabla_{e_i}\nabla f,e_i \>\\&=\sum_{i,j=1}^n 
\<\right.\hskip-.02in \nabla_{e_i}\! (\<e_j,\x\> e_j),e_i\hskip-.02in \left.\>
\\& = \sum_{i,j=1}^n \left(  \<\right.\hskip-.02in \tilde\nabla_{e_i}e_j,\x \hskip-.02in\left. \>\< e_j,e_i\>+\<e_j,e_i\>^2
+\<e_j,\x\>\<\nabla_{e_i}e_j,e_i\>\right)
\\& =n+ n\<H,\x\>+\sum_{i,j=1}^n \left(   \<\nabla_{e_i}e_j,\x  \>\< e_j,e_i\>+\<e_j,\x\>\<\nabla_{e_i}e_j,e_i\>\right),
\end{aligned}\end{equation}
where we have applied the well-known fact: $\tilde\nabla_X \x=X$ for any vector $X$ tangent to $M^n$.

If we put
\begin{align}\label{3.4}\nabla_X e_i=\sum_{k=1}^n \omega_i^k(X)e_k,
\end{align}
we get $\omega_i^k=-\omega_k^i$ for $1\leq i,k\leq n$.
Thus we obtain 
\begin{equation}\begin{aligned} \label{3.5}
\sum_{i,j=1}^n &\left(   \<\nabla_{e_i}e_j,\x  \>\< e_j,e_i\>+\<e_j,\x\>\<\nabla_{e_i}e_j,e_i\>\right)\\& =\sum_{i,k=1}^n \omega_i^k(e_i)\<e_k,\x\>+\sum_{i,j=1}^n \omega_j^i(e_i)\<e_j,\x\>
\\& =0.\end{aligned}\end{equation}
After  combining \e{3.3} and \e{3.5} we find
\begin{align}  \label{3.6}
{\rm div}(\x^T)=n(1+\<H,\x\>).
\end{align}
Therefore the canonical vector field $\x^T$ is incompressible if and only if $\<H,\x\>=-1$  holds identically. Consequently, we obtain statement (2).
\end{proof}

An immediate consequence of Theorem \ref{T:3.1} is the following.

\begin{corollary}\label{C:3.1} Let $M^n$ be an $n$-dimensional submanifold of the Euclidean $m$-space $\mathbb{E}^m$. Then the canonical vector field $v$ of $M^n$ is incompressible if and only if $\<\x,\Delta \x\>=n$  holds identically on $M^n$.
\end{corollary}
\begin{proof} Follows from  Theorem \ref{T:3.1}(2) and the well-known formula of Beltrami: $\Delta\x=-nH$ (see, e.g., \cite[page 41]{book11}).
\end{proof} 

Let $N$ be an $(n-1)$-dimensional submanifold of the unit hypersphere $S^{m-1}_o(1)$ of $\mathbb E^m$ centered at the origin $o\in \mathbb E^m$. The {\it cone over $N$ with vertex at $o$}, denoted by $CN$, is defined by the following map: $$N\times (0,\infty)\to \mathbb E^m; (p,t)\mapsto t p.$$

Another immediate consequence of Theorem \ref{T:3.1} is the following.

\begin{corollary}\label{C:3.2} The canonical vector field of any cone with vertex at the origin of $\mathbb{E}^m$ is never incompressible.
\end{corollary}
\begin{proof} Follows from Theorem \ref{T:3.1}(2) and  fact that the position vector field of any cone with vertex at $o\in \mathbb{E}^m$ is tangent to the cone.
\end{proof}

If $(N,g)$ is a compact Riemannian homogeneous manifold, denote by $G$ the identity component of the group of isometries of $N$. Then $G$ is a compact Lie group which acts transitively on $N$. Thus $N = G/K$, where $K$ is the isotropy subgroup of $G$ at a point $p\in N$.
An immersion $\phi : (N, g) \to \mathbb E^m$ is said to be {\it equivariant} if and only if there exists a Lie homomorphism $\psi : G \to SO(m)$ such that $\phi(q(p)) = \psi(q)(\phi(p))$ for any $q \in G$ and $p\in N$.

\begin{theorem}\label{T:3.2} Every equivariantly isometrical immersion of a compact homogeneous Riemannian manifold into any Euclidean space has incompressible canonical vector field.\end{theorem}
\begin{proof} Let $N$ be a compact Riemannian homogeneous manifold and $\phi : N \to \mathbb E^m$ be an equivariantly isometrical immersion. Then $1+\<H,\x\>$ is a constant, say $c$, on $N$, where $\x$ and $N$ denote the position vector field and mean curvature vector field of $N$ in $\mathbb E^m$, respectively. Hence we have
\begin{align}\label{3.7} \int_M (1+\<H,\x\>)dV= c \,{\rm vol}(N),\end{align}
where $dV$ and ${\rm vol}(N)$ denote the volume element of volume of $N$, respectively.

On the other hand, we have the following formula of Minkowski-Hsiung's formula \cite[page 305]{book15}:
\begin{align}\label{3.8} \int_M (1+\<H,\x\>)dV= 0.\end{align}
By comparing \e{3.7} and \e{3.8}, we obtain $\<H,\x\>=-1$ identically on $N$. Consequently, the equivariant immersion has incompressible canonical vector field according to Theorem \ref{T:3.1}(2).
\end{proof} 

An immediate consequence of Theorem \ref{T:3.2} is the following.

\begin{corollary}\label{C:3.3} Every hypersphere centered at the origin of $\mathbb E^{n+1}$ has incompressible canonical vector field.
\end{corollary}

\section{Planar curves with incompressible canonical vector field}

Recall that, up to rigid motions, a unit speed planar curve is determined completely by its curvature function.

The next theorem completely classifies planar curves with incompressible canonical vector field.

\begin{theorem}\label{T:4.1} Let $\gamma(s)$ be a unit speed planar curve in $\mathbb E^2$. Then
the canonical vector field of $\gamma$ is incompressible if and only if, up to rigid rotations of $\mathbb E^2$ about the origin, $\gamma$ is an open portion of a curve of the following two types:
\begin{itemize}
\item[\textrm{(a)}] A circle centered at the origin;

\item[\textrm{(b)}] A curve defined by
\begin{align} \label{4.1}\gamma= \frac{2}{c^2}\Big(\! \cos(c\sqrt{s})+c\sqrt{s}\sin (c\sqrt{s}),
 \sin(c\sqrt{s})-c\sqrt{s}\cos (c\sqrt{s})\Big)
\end{align}
for some nonzero real number $c$.
\end{itemize}
 \end{theorem}
\begin{proof} For the unit speed planar curve $\gamma(s)$ in $\mathbb E^2$, we have
\begin{align}\label{4.2} &\gamma'(s)=T, \;\; T'(s)= \kappa N(s),\;\; N'(s)=-\kappa T,\\&\label{4.3} H=\kappa N,
\end{align}
where $T$ and $N$ are the unit tangent vector field and the principal normal, respectively, and $\kappa$ is the curvature function of $\gamma$.

Now, let us assume that $\gamma$ is planar curve whose canonical vector field is incompressible.
Then it follows from \e{4.3} and Theorem 3.1(2) that 
\begin{align} \label{4.4}\<N,\gamma\>=-\frac{1}{\kappa}.
\end{align}
By differentiating \e{4.2} with respect to the arclength $s$, we get
\begin{align} \label{4.5}\<T,\gamma\>=-\frac{\kappa'}{\kappa^3}.
\end{align}

\vskip.1in
{\it Case} (a): If $\kappa'(s)=0$. In this case, \e{4.5} implies that $\<T,\gamma\>=0$ holds identically. Thus the position vector field of $\gamma$ is always normal to the curve. Hence
$\gamma$ is an open part of a circle centered at the origin of $\mathbb E^2$.

Conversely, if  $\gamma$ is an open part of a circle centered at the origin of $\mathbb E^2$, then the canonical vector field $\x^T$ of $\gamma$ vanishes identically. Therefore $\x^T$ is trivially incompressible.

\vskip.1in
{\it Case} (b): If $\kappa'(s)\ne 0$. In this case, by differentiating \e{4.5} with respect to $s$ and using \e{4.4}, we find
\begin{align} \label{4.6}0=\(\frac{\kappa'}{\kappa^3}\)',
\end{align}
which implies that the curvature function $\kappa(s)$ satisfies
\begin{align} \label{4.7}\frac{\kappa'}{\kappa^3}=-\frac{2}{b}
\end{align} for some nonzero real number $b$.

After solving \e{4.7} we obtain
\begin{align} \label{4.8}\kappa^2=\frac{b}{4(s-a)}
\end{align}
for some real number $a$.  Therefore, by applying a suitable translation on $s$, we get
\begin{align} \label{4.9}\kappa^2=\frac{b}{4s}
\end{align}
Without loss of generality, we may assume that $b$ is positive. So we may put $b=c^2$ and we get
\begin{align} \label{4.10} \kappa=\frac{c}{2\sqrt{s}}.
\end{align}
Therefore, up to rigid rotations of $\mathbb E^2$ about the origin, the planar curve $\gamma$ is given by
\begin{align} \label{4.11}\gamma= \frac{2}{c^2}\Big(\! \cos(c\sqrt{s})+c\sqrt{s}\sin (c\sqrt{s}),
 \sin(c\sqrt{s})-c\sqrt{s}\cos (c\sqrt{s})\Big)
\end{align}
which is exactly \e{4.1}.

Conversely, if the curve $\gamma$ is defined by \e{4.1}, then we get 
\begin{equation}\begin{aligned} \notag& \gamma'(s)=\big(\cos(c\sqrt{s}), \sin(c\sqrt{s})\big),\;\;
\\& H=\gamma''(s)=\frac{-c}{2\sqrt{s}}\big(\sin(c\sqrt{s}), - \cos(c\sqrt{s})\big).
\end{aligned}\end{equation}
It is direct to verify that $\<H,\gamma\>=-1$ holds identically. Consequently, the canonical vector field of $\gamma$ is incompressible  according to statement (2) of Theorem \ref{T:3.1}.
\end{proof}

\vskip-0.3in

  \begin{figure}[h]  
\setlength{\unitlength}{1.4cm}  
\begin{minipage}{5cm}  
 \begin{picture}(5,5)  
\scalebox{0.25}{\includegraphics{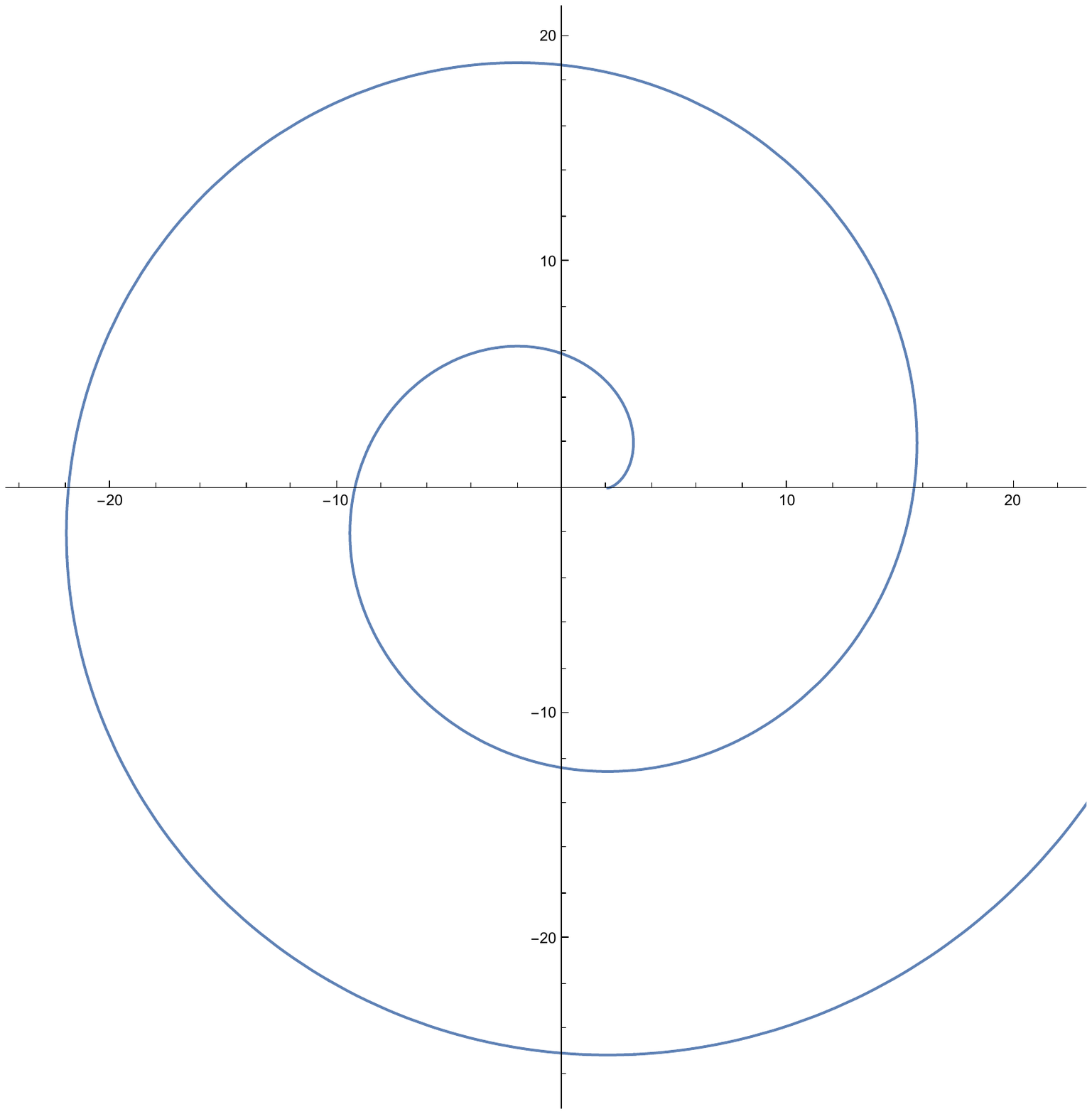} }
\end{picture}  
\end{minipage}  
 \end{figure} 
\vskip-.7in
 
 \centerline{{\bf Figure 1.} Planar curve with incompressible canonical}

 \centerline{vector field defined by (4.1) with $c=1$ and $s\in (0,60\pi)$.}

\section{Hypercylinders and surfaces of revolution with incompressible canonical vector field}

In this section, we present the next result which classifies hypercylinders with incompressible canonical vector field.

\begin{theorem} Let $M^n$ $(n\geq 2)$ be a hypercylinder over a unit speed planar curve $\gamma(s)\subset \mathbb E^2\subset \mathbb E^{n+1}$. Then the canonical vector field of $M^n$ is incompressible if and only if, up to rotations  of $\mathbb E^{2}$ about the origin of $\mathbb E^2$, $M^n$ is an open part of the hypersurface defined by
\begin{equation}\begin{aligned}\label{5.1} &\Bigg((1-n) s  \cos (K(s))+ \sqrt{n(n-1)}\sqrt{c^2-s^2}\sin (K(s)),
\\&\hskip.4in (1-n) s  \sin (K(s))- \sqrt{n(n-1)}\sqrt{c^2-s^2}\cos (K(s)),\\&\hskip.8in  t_2,\ldots,t_n\Bigg),
\end{aligned}\end{equation}
where $K(s)=\frac{\sqrt{n}}{\sqrt{n-1}}\arctan\({s}/{\sqrt{c^2-s^2}}\)$ with $-c<s<c$.
\end{theorem}
\begin{proof}
Let $M^n\subset \mathbb E^{n+1}$ be a hypercylinder over a unit speed planar curve $\gamma(s)$ with $n\geq 2$. Without loss of generality, we may assume that $M^n$ is parametrized by
\begin{align}\label{5.2} x(s,t_2,\ldots,t_n)=\big(\gamma(s),t_2,\ldots,t_n),
\end{align}
where $\gamma(s)$ is a unit speed planar curve. It is easy to verify that the mean curvature vector of $M^n$ in $\mathbb E^{n+1}$ satisfies
\begin{align}\label{5.3} H=\frac{\gamma''}{n}=\frac{\kappa N}{n},
\end{align}
where $N$ denotes the principal normal of the planar curve $\gamma$.
Hence it follows from \e{5.2}, \e{5.3} and  Theorem 3.1(2) that the hypercylinder $M^n$ has incompressible canonical vector field if and only if the planar curve $\gamma$ satisfies
\begin{align}\label{5.4} \<\gamma,\gamma''\>=-n.\end{align}

Analogous to the proof Theorem \ref{T:4.1}(b), it is direct to verify  that if  a unit speed planar curve $\gamma$ satisfies \e{5.4}, then its curvature function $\kappa$ satisfies
\begin{align}\label{5.5}n(\kappa\kappa''-3\kappa'^2)=(n-1)\kappa^4,\;\; n\geq 2.\end{align}

After solving this second order differential equation, we know that up to translations on $s$ the solutions of \e{5.5} satisfy
\begin{align}\label{5.6}\kappa(s)=\frac{\sqrt{n}}{\sqrt{n-1}\sqrt{c^2-s^2}},\;\; n\geq 2,
\end{align}
 where $c$ is a nonzero real number. From \e{5.6} we obtain
\begin{align}\label{5.7}K(s)=\frac{\sqrt{n}}{\sqrt{n-1}}\arctan\(\frac{s}{\sqrt{c^2-s^2}}\),   \;\; -c<s<c,
\end{align} where $K$ is an anti-derivative of $\kappa$.
Hence, up to rotations of $\mathbb E^2$ about the origin, the planar curve $\gamma(s)$ with curvature function $\kappa$ is given by
\begin{align}\label{5.8}\gamma=\(\int^s\cos (K(s))ds,\int^s \sin (K(s))ds\).
\end{align}
It follows from \e{5.8} that 
\begin{align}\label{5.9}\gamma''=-\kappa(s)\Big(\sin (K(s)), -\cos (K(s))\Big).
\end{align}
Now, it follows from \e{5.4}, \e{5.8} and \e{5.9} that
\begin{align}\label{5.10} \frac{n}{\kappa(s)}=\sin (K(s))\! \int^s\! \cos (K(s))ds-\cos (K(s))\! \int^s\! \sin (K(s))ds,
\end{align}
which yields
\begin{equation}\begin{aligned}\label{5.11} \int^s\cos (K(s))ds=\,&\sqrt{n(n-1)}\sqrt{c^2-s^2}\csc (K(s))
\\&+ \cot (K(s))\int^s \sin (K(s))ds.
\end{aligned}\end{equation}
Now, by differentiating \e{5.11} we find
\begin{equation}\begin{aligned}\label{5.12} \int^s\sin (K(s))ds=\,&(1-n) s  \sin (K(s))\\&- \sqrt{n(n-1)}\sqrt{c^2-s^2}\cos (K(s)).
\end{aligned}\end{equation}
By substituting \e{5.12} into \e{5.10} we find 
\begin{equation}\begin{aligned}\label{5.13} \int^s\cos (K(s))ds=\,&(1-n) s  \cos (K(s))\\&+ \sqrt{n(n-1)}\sqrt{c^2-s^2}\sin (K(s)).
\end{aligned}\end{equation}

Therefore, by substituting \e{5.10} and \e{5.11} into \e{5.8} we obtain
\begin{equation}\begin{aligned}\label{5.14} \gamma=\,&\Bigg((1-n) s  \cos (K(s))+ \sqrt{n(n-1)}\sqrt{c^2-s^2}\sin (K(s)),
\\&\hskip.1in (1-n) s  \sin (K(s))- \sqrt{n(n-1)}\sqrt{c^2-s^2}\cos (K(s))\Bigg).
\end{aligned}\end{equation}
It is easy to verify that the planar curve $\gamma$ defined by \e{5.14} satisfies $\<\gamma,\gamma''\>=-n$ identically. Consequence, the hypercylinder over $\gamma$ with incompressible canonical vector field is an open part of the hypersurface defined by \e{5.1}.
\end{proof}

  \begin{figure}[h]  
\setlength{\unitlength}{1.4cm}  
\begin{minipage}{5.5cm}  
 \begin{picture}(5,5)  
\scalebox{0.28}{\includegraphics{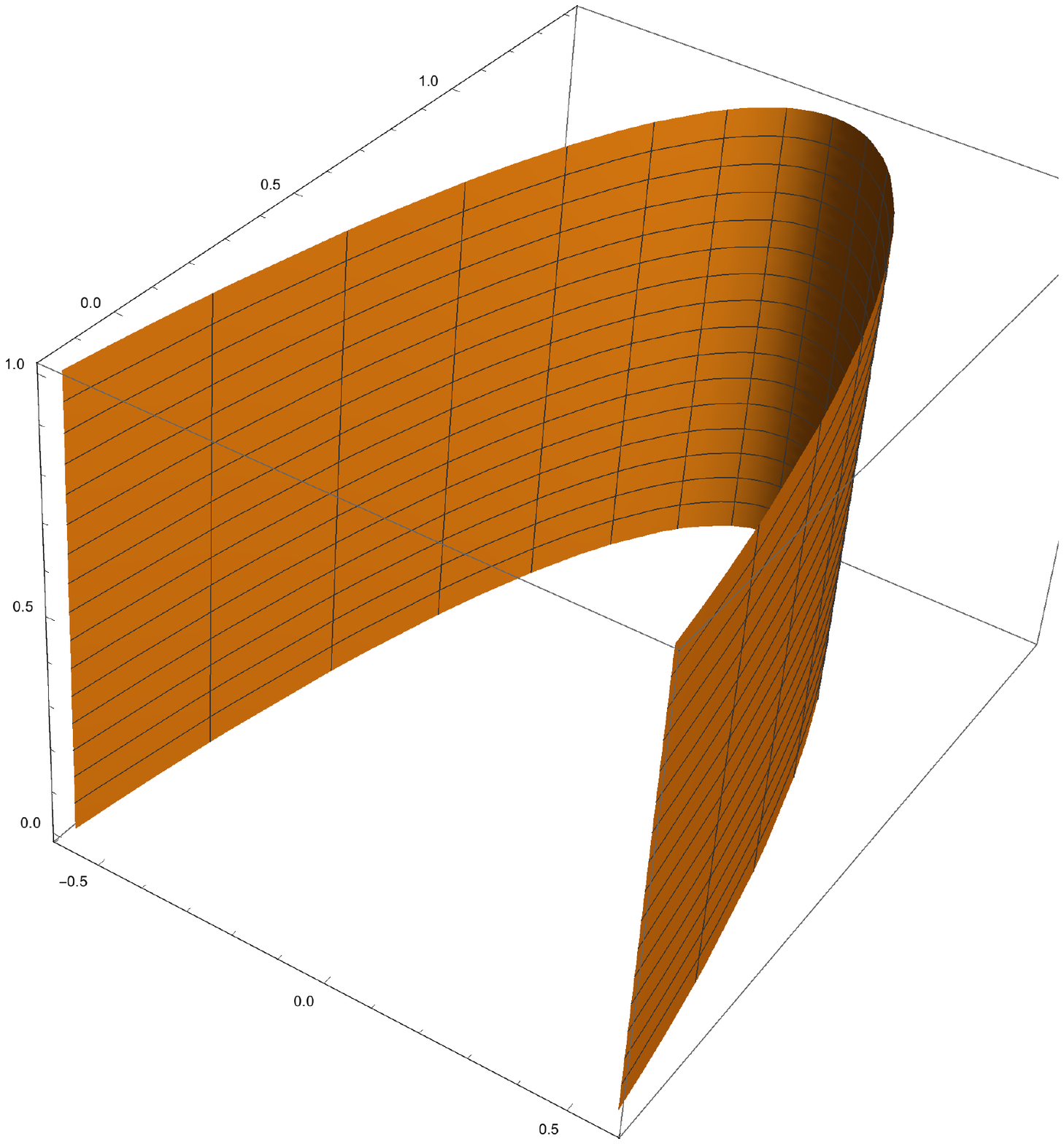} }
\end{picture}  
\end{minipage}  
 \end{figure} 
\vskip-0.5in
 
 \centerline{{\bf Figure 2.} Hypercylinder with incompressible canonical vector field}

 \centerline{defined by (5.1) with $c=1$, $n=2$, $|s|<0.7$ and $t_2\in (0,1)$.}

\vskip.2in
If we consider a {\it surface of revolution} in $\mathbb E^3$ of the form:
\begin{align}\label{5.15} \x(s,t)=(r(s)\cos t, r(s)\sin t, s), \end{align}
then it is direct to verify that the mean curvature vector of the surface is given by
\begin{equation}\begin{aligned}\label{5.16} H=\,&  \frac{1+r'(s)^2-r(s)r''(s)}{2r(s)(1+r'(s)^2)^2}\big(\! -\cos t,-\sin t,r'(s)\big).
\end{aligned}\end{equation}
It follows from \e{5.15}, \e{5.16} and Theorem \ref{T:3.1}(2) the following.

\begin{proposition} The surface of revolution in $\mathbb E^3$ defined by \e{5.15} has incompressible canonical vector field if and only if the function $r$ satisfies the following second order differential equation:
\begin{align}\label{5.17} (1+r'^2)(r+sr'+2r r'^2)+r(r-sr')r''=0.\end{align}
\end{proposition}

\begin{remark} One solution of \e{5.17} is $r=\sqrt{1-s^2}$. The corresponding surface of revolution defined by \e{5.15} with $r=\sqrt{1-s^2}$ is nothing but the unit sphere centered at the origin in $\mathbb E^3$.
\end{remark}

\section{Examples of surfaces in $\mathbb E^4$ with incompressible canonical vector field}

Finally, we provide some examples of surfaces in $\mathbb E^4$ with incompressible canonical vector field.
Let us consider product surfaces of two unit speed planar curves $\beta(s)$ and $\gamma(t)$ given by
\begin{align}\label{6.1} \x(s,t)=(\beta(s), \gamma(t)).\end{align}
Then the mean curvature vector of this product surface is given by
\begin{align}\label{6.2} H=\frac{1}{2} \big(\beta''(s), \gamma''(t)\big).\end{align}
According to Theorem \ref{T:3.1}(2), the surface has incompressible canonical vector field if and only if
\begin{align}\label{6.3} \<\beta(s),\beta''(s)\>+\<\gamma(t), \gamma''(t)\>=-2\end{align}
holds identically, which implies that $\beta(s)$ and $\gamma(t)$ satisfy
\begin{equation}\begin{aligned}\label{6.4} &\<\beta(s),\beta''(s)\>=-a,\;\; \\& \<\gamma(t), \gamma''(t)\>=-(2-a)\end{aligned}\end{equation}
for some constant $a$.

The simplest such examples in $\mathbb E^4$ with incompressible canonical vector field is by taking $a=1$ in \e{6.4}.  Clearly, such a surface is the product of two planar curves with incompressible canonical vector field given by type (a) or type (b) as defined in Theorem \ref{T:4.1}.

\end{document}